\documentclass{amsart}
\usepackage{amsmath,amsfonts,amssymb,amsthm,amscd,latexsym}

\usepackage{tikz}
\usetikzlibrary{arrows}
\usepackage{color}
\usepackage{cite}
\usepackage{multirow}
\usepackage{cite}

 \usepackage{hyperref}					

\newtheorem{thm}{Theorem}[section]
\newtheorem{theorem}[thm]{Theorem}
\newtheorem{corollary}[thm]{Corollary}

\newtheorem{lemma}[thm]{Lemma}
\newtheorem{prop}[thm]{Proposition}
\newtheorem{defn}[thm]{Definition}
\newtheorem{rem}[thm]{Remark}
\newtheorem{exam}[thm]{Example}

\newtheorem{prob}[thm]{Problem}

\numberwithin{equation}{section}
\newcommand{\cA}{{\mathcal A}}

\newcommand{\cD}{{\mathcal D}}

\newcommand{\cM}{{\mathcal M}}

\newcommand{\cP}{{\mathcal P}}

\newcommand{\cR}{{\mathcal R}}
\newcommand{\cS}{{\mathcal S}}

\usepackage{mathabx,graphicx}

\def\CircleArrowright{\ensuremath{%
  \reflectbox{\rotatebox[origin=c]{180}{$\circlearrowright$}}}}

\begin{document}

\title[Non-existence of  translation-invariant derivations]{Non-existence of translation-invariant derivations on algebras of   measurable functions}

\author[Ber]{Aleksey Ber}
\address{Department of Mathematics, National University of Uzbekistan, Vuzgorodok,
100174, Tashkent, Uzbekistan}
\email{aber1960@mail.ru}

\author[Huang]{Jinghao Huang}
\address{School of Mathematics and Statistics, University of New South Wales,
Kensington, 2052, Australia}
\email{jinghao.huang@unsw.edu.au}

\author[Kudaybergenov]{Karimbergen Kudaybergenov}
\address{Department of Mathematics, Karakalpak State University, Ch. Abdirov 1, Nukus 230113, Uzbekistan}
\email{karim2006@mail.ru}

\author[Sukochev]{Fedor Sukochev}
\address{School of Mathematics and Statistics, University of New South Wales,
Kensington, 2052, Australia}
\email{f.sukochev@unsw.edu.au}

\begin{abstract}
Let $S(0,1)$ be the $*$-algebra of all classes of  Lebesgue measurable functions on the unit interval $(0,1)$
and let $(\cA,\left\|\cdot \right\|_\cA)$ be a complete  symmetric  $\Delta$-normed  $*$-subalgebra of $S(0,1)$, in which simple functions are dense, e.g., $L_\infty (0,1)$, $L_{\log}(0,1)$, $S(0,1)$ and the Arens algebra $L^\omega (0,1)$ equipped with their natural $\Delta$-norms.
We show that there exists no non-trivial derivation $ \delta: \cA \to S(0,1)$ commuting  with all dyadic translations of the unit interval.
Let $\cM$ be a type $II$ (or $I_\infty$) von Neumann algebra, $\cA$ be its  abelian von Neumann subalgebra, let  $S(\cM)$ be the algebra  of all measurable operators affiliated with $\cM$.
 We show that any  non-trivial derivation $\delta:\cA\to S(\cA)$ can not be extended to a derivation on $S(\cM)$. In particular, we  answer an  untreated question  in \cite{BKS1}.
\end{abstract}

\subjclass[2010]{46L57, 47B47, 46L51, 26A24. \hfill Version~: \today.}

\keywords{Murray--von Neumann algebras; derivations; approximately differentiable functions; dyadic translations.}

\maketitle


\section{Introduction}

Let
$\mathcal{A}$ be an algebra over the field of complex numbers.
 A linear operator
$\delta:\mathcal{A}\rightarrow \mathcal{A}$ is called a
\textit{derivation} if $\delta$ satisfies the Leibniz rule, i.e.,  $\delta(xy)=\delta(x)y+x\delta(y)$ for all $x, y\in
\mathcal{A}$.
The theory of derivations  is an important and well studied part of the general theory
 of operator algebras, with significant applications in mathematical physics (see, e.g.~\cite{Bra}, \cite{Sak2}).
One of the most  important examples of derivations is  the usual differential operator $\frac{d\cdot}{dt}$ on the algebra $D(0,1)$ consisting of all classes in $S(0,1)$ which contain    functions  having finite derivative almost everywhere in   $(0,1).$
In \cite{BSCh06} (see also \cite{BSCh04}),
 the problem of existence of non-trivial
derivations   in the setting  of von Neumann regular
commutative algebras was considered.
As an  application, it  was established in \cite{BSCh06}  that the algebra $S(0, 1)$ of all measurable complex
functions on the interval $(0, 1)$ admits non-trivial derivations \cite[Theorem 3.1]{BSCh06} (see also  \cite[Remark 6.3]{Kusraev} for an alternative proof).
 In particular, it is established in \cite[Theorem 3.1]{BSCh06} that
 there exist  derivations on the algebra $S(0,1)$ of all classes of   measurable functions on~$(0,1)$ which extend the derivation
 $\frac{d\cdot }{dt}$ on the algebra $D(0,1)$.
A natural question is
\begin{quote}
\emph{what properties of $\frac{d\cdot}{dt}$ on $D(0,1)$   are shared by its  extension?}
\end{quote}


A very important property of $\frac{d\cdot}{dt}$  is the translation-invariance  property, which has been widely studied since the 1970s.
In particular, S. Sakai \cite[Proposition 1.17]{Sakai77}  proved that a   closed
derivation on $C(\mathbb{T})$ commuting with translations by elements of $\mathbb{T}$ is a constant
multiple of $\frac{d\cdot}{dt }$, where $\mathbb{T}$ is the one-dimensional torus.

Let $x\in S(0,1)$.
Set
\begin{align}\label{defalpha}
(\alpha_n(x))(t) =x\left(\left\{t - \frac{1}{2^n}\right\}\right), \qquad n\in \mathbb{N},
\end{align}
where $\{t\}$ stands for the fractional part of the number $t$.
Note $\{\alpha_n\}$ generates   the group $G$ of dyadic-rational  translations of $S(0,1)$ (see Section \ref{3.2} for the definition), 
which is a
subgroup   in  the group $Aut(S(0,1))$ of all automorphisms of $S(0,1)$.

It is well-known (see Section \ref{AP}) that  the (approximately-)differential operator~$\frac{d\cdot }{dt}$ (respectively, $\partial_{AD}$) is translation-invariant on $D(0,1)$ (respectively, the algebra $AD(0,1)$ of all approximately differentiable functions on $(0,1)$).
We are interested in studying the translation-invariance property of derivations on the larger algebra  $S(0,1)$:
\begin{equation*}
 \begin{tikzpicture}
\node at (7 ,7) {$D(0,1)\subset AD(0,1)\ \subset \ \  S(0,1)\ \,\, \   $};
\node at (7,5) {$S(0,1) ~= \     S(0,1) \ \ \subset \ \   S(0,1)\,\   \ $};
\node at (6,7.7) {$  \alpha_n ~ \ \  \ \ ~ \ \ ~ \ ~ \ \ \  \alpha_n   ~ \,\ ~ \ ~$};
\node at (6,7.4) {$  \CircleArrowright  ~  \ \ \ \ ~ \ \ ~ \ ~ \ \ \  \CircleArrowright   ~ \,\ ~ \ ~$};
\node at (6,4.6) {$  \circlearrowleft ~ \ \  \ \ \ ~ \ ~ \ ~ \ \ \ \ \circlearrowleft  ~ \,\ ~ \ ~$};
\node at (6,4.3) {$ \alpha_n ~ \ \ \  \ \ ~ \ ~ \ ~ \ \ \ \ \alpha_n  ~ \,\ ~ \ ~$};
\node at (4.8,6) {$\frac{d}{dt}$};
\node at (5.8,6) {$\subset$};
\node at (6.9,6) {$\partial_{AD}$};
\node at (7.8,6) {$\subset$};
\node at (8.4,6) {$\delta$};
\draw [-stealth] (5,6.7)--(5,5.3);
\draw [-stealth] (6.5,6.7)--(6.5,5.3);
\draw [-stealth] (8.5,6.7)--(8.5,5.3);
\end{tikzpicture}
\end{equation*}
where $\delta|_{AD(0,1)}=\partial _{AD}$ and $\partial_{AD}|_{D(0,1)}=\frac{d\cdot }{dt}$.

The  main result of the present paper  is an interesting   property of derivations on~$S(0,1)$,
which shows that non-trivial derivations on $S(0,1)$ do not commute with all $\alpha_n$.
This is in strong contrast with the result by Sakai\cite{Sakai77}. 


\begin{theorem}\label{th11}
Let $(\cA, \left\|\cdot\right\|_\cA)$ be a complete symmetric $\Delta$-normed   $*$-subalgebra of $S(0,1 )$
in which simple functions are dense in the $\left\|\cdot \right\|_\cA$-norm topology.
Let $\delta$ be a derivation from $\cA$ into $S(0,1 )$ commuting with all $\alpha_n$, $n\in \mathbb{N}$.   Then  $\delta$ is trivial.
In particular, the approximately differential operator
$\partial_{AD}$
has no translation-invariant extension as a derivation on the  algebra $S(0,1)$.
\end{theorem}

In \cite{BKS1}, a  noncommutative analogue $AD(\cR)$
of the algebra of all almost everywhere approximately
differentiable functions on $(0,1)$ for the hyperfinite
type $II_1$ factor~$\cR$  was  introduced (all necessary definitions can be found in Section \ref{AP} below or in  \cite{BKS1}).
It was also established that the classical approximately  differential operator  on the algebra  $AD(0,1)$ admits an extension to a derivation $\delta$ from $AD(\mathcal{R})$ into $S(\mathcal{R})$ with $\delta|_{AD(0,1)} =\partial_{AD}$,
where $AD(0,1)$ can be viewed as 
 a subalgebra of $S(\cR)$ (see Section \ref{AP} below or \cite{BKS1}).

In \cite{BKS1}, the question  whether  the approximately differential operator $\partial_{AD}$ on $AD(\cR)$ has  an extension to the algebra $S(\cR) $
 was left unanswered.
Now using the main result of the paper \cite{BKS2},
we are able to answer this question.

\begin{prop}\label{1.2}
Let $\cM$ be a type $II$ (or $I_\infty$) von Neumann algebra, $\cA$ be its  abelian von Neumann subalgebra. Suppose that $\delta$ is a  derivation on $S(\cM)$ such that the range  $\delta|_{\cP(\cA)}$ of the projection lattice $\cP(\cA)$ of $\cA$ is contained in  $ S(\cA).$ Then $\delta$ vanishes on $S(\cA).$
In particular,
\begin{enumerate}
  \item[(i).]  any  non-trivial derivation $\delta:\cA\to S(\cA)$ can not be extended to a derivation on $S(\cM)$;
  \item[(ii).] there exists no derivation $\delta:S(\cR) \to S(\cR)$ such that $\delta |_{AD(0,1)} =  \partial_{AD},$ where  $\cR$ is  the   type $II_1$ hyperfinite factor.
\end{enumerate}

\end{prop}

\section{Preliminaries}\label{sec2}

In this section,  we briefly list some necessary facts concerning algebras of measurable operators.

Let $H$  be a Hilbert space and let $B(H)$ be the $\ast$-algebra of all bounded linear operators
on $H.$ A von Neumann algebra $\mathcal{M}$  is a weakly closed unital $\ast$-subalgebra in $B(H).$ For details on von Neumann algebra theory, the reader is
 referred to \cite{Dixmier, KR1, KRII, SS, Tak}. General facts concerning measurable operators may be found in \cite{Nel,Segal} (see also \cite[Chapter IX]{Tak2} and the forthcoming book \cite{DPS}). For convenience of the reader, some of the basic definitions are recalled below.

\subsection{Murray-von Neumann algebras}\label{MvN}

Let $\cM$ be a semifinite von Neumann algebra.
A densely defined closed linear operator $x : \textrm{dom}(x) \to  H$
(here the domain $\textrm{dom}(x)$ of $x$ is a linear subspace in $H$) is said to be \textit{affiliated} with $\mathcal{M}$
if $yx \subset  xy$ for all $y$ from the commutant $\mathcal{M}'$  of the algebra~$\mathcal{M}.$

Denote by $P(\mathcal{M})$ the set of all projections in $\mathcal{M}.$ Recall that two projections $e, f \in  P(\mathcal{M})$ are called \textit{equivalent} (denoted by $p\sim q$) if there exists an element
$u \in \mathcal{M}$ such that $u^\ast  u = e$ and $u u^\ast  = f.$
For projections $e, f \in  \mathcal{M}$,
the notation $e \preceq  f$ means that there exists a projection $q \in  \mathcal{M}$ such that
$e\sim q \leq f.$ A projection $p \in \mathcal{M}$ is called \textit{finite}, if the
conditions $q \leq  p$ and $q$ is equivalent to $p$ imply that $q = p.$
If the unit $\mathbf{1}$ of the von Neumann algebra $\mathcal{M}$  is a finite projection in $\cM$, then $\mathcal{M}$ is called {\it finite.}

A linear operator $x$ affiliated with $\mathcal{M}$ is called \textit{measurable} with respect to $\mathcal{M}$ if
$\chi_{(\lambda,\infty)}(|x|)$ is a finite projection for some $\lambda>0.$ Here
$\chi_{(\lambda,\infty)}(|x|)$ is the  spectral projection of $|x|$ corresponding to the interval $(\lambda, +\infty).$

The development of  non-commutative integration theory was initiated by Murray and von Neumann \cite{MvN1} and by  Segal~\cite{Segal}, who introduced new classes of (not necessarily Banach) algebras of unbounded operators, in particular the algebra  $S(\mathcal{M})$ of all measurable operators affiliated with a von Neumann algebra $\mathcal{M}$.
The specific interest of the study of $S(\cM)$ when $\cM$ is a $II_1$ von Neumann algebra,  is also recorded
  von Neumann's   talk at  the International Congress of Mathematicians, Amsterdam, 1954 \cite[p.231--246]{RS}.
In the special case  when $\cM$  is   a finite von Neumann algebra, the algebra $S(\cM)$ of all densely defined closed operators affiliated with $\cM$ is frequently  referred as the \emph{Murray-von Neumann algebra} associated with $\cM$, which  is the algebra of all densely defined closed operators affiliated with $\mathcal{M}$ (see e.g. \cite{KL,KLT}).

Let $x, y \in  S(\mathcal{M}).$ It is well known that $x+y$ and
$xy$ are densely-defined and preclosed
operators. Moreover, the (closures of) operators $x + y, xy$ and $x^\ast$  are also in $S(\mathcal{M}).$
When
equipped with these operations, $S(\mathcal{M})$ becomes a unital $\ast$-algebra over $\mathbb{C}$  (see \cite{Ciach}). It
is clear that $\mathcal{M}$  is a $\ast$-subalgebra of $S(\mathcal{M}).$



From now on, we shall always assume that
  $\cM$ is  a finite von Neumann algebra equipped with a  faithful normal finite trace $\tau$.
Consider the    \textit{measure topology} $t_{\tau}$
on $S(\mathcal{M}),$ which is defined by
the following neighborhoods of zero:
$$
N(\varepsilon, \delta)=\{x\in S(\mathcal{M}): \exists \, e\in P(\mathcal{M}), \, \tau(\mathbf{1}-e)\leq\delta, \, xe\in
\mathcal{M}, \, \|xe\|_\infty\leq\varepsilon\},
$$
where $\varepsilon, \delta$
are positive numbers,    $\mathbf{1}$ is the unit in $\mathcal{M}$ and $\left\| \cdot \right \|_\infty$ denotes the operator norm on
$\mathcal{M}.$  The algebra $S(\mathcal{M})$ equipped with the measure topology is a topological algebra.

Let $x\in S(\mathcal{M})$ and let  $x=v|x|$ be the polar decomposition of $x$ \cite{DP2,LSZ}.
Then $l(x) =v v^\ast$ and  $r(x)=v^\ast v$ are left and right supports of the element  $x$, respectively.

We define the so-called rank  metric $\rho$ on $S(\mathcal{M})$  by setting
$$
\rho(x, y)=\tau( r((x-y)))=\tau(l(x-y)),\,\, x, y\in \mathcal{A}.
$$
 In fact, the rank-metric $\rho$ was firstly introduced in a general case of regular rings by von Neumann  in \cite{Neu37}, where it was shown that $\rho$ is  a metric on $S(\cM)$.
By~\cite[Proposition 2.1]{Ciach}, the algebra $S(\mathcal{M})$ equipped with the metric $\rho$ is a complete topological ring.
We note that if $\{x_n\} _{n=1}^\infty $ is a sequence of self-adjoint operators in $S(\cM)$ having pairwise orthogonal  supports, then $\sum_{n=1}^\infty  x_n$ exists in the topology induced by $\rho$ and also in measure.
In the special case when $\cM=L^\infty (0,1)$, the metric $\rho$ on the regular algebra $S(\cM)=S(0,1)$ is the same as in \cite{BSCh06,BKS2}.

\subsection{Symmetrically $\Delta$-normed  spaces of  measurable functions}\label{2.2}
For convenience of the reader, we recall the definition of $\Delta$-norms.
Let $E$ be a linear space over the field $\mathbb{C}$.
A function $\left\|\cdot\right\|$ from $E$ to $\mathbb{R}$ is a $\Delta$-norm, if for all $x,y \in E$ the following properties hold:
\begin{align*}
\left\|x\right\| \geqslant 0 , ~\left\|x\right\| = 0 \Leftrightarrow x=0 ;\\
\left\|\alpha x\right\| \leqslant \left\|x\right\|, ~\forall~\alpha \in \mathbb{C},  |\alpha| \le 1 ;\\
\lim _{\alpha \rightarrow 0}\left\|\alpha x\right\| = 0;\\
\left\|x+y \right\| \le C_E \cdot (\left\|x\right\|+\left\|y\right\|)
\end{align*}
 for a constant $C_E\geq 1$ independent of $x,y$.
The couple $(E, \left\|\cdot\right\|)$ is called a \emph{$\Delta$-normed} space.
We note that the definition of a $\Delta$-norm given above is the same as  in \cite{KPR}.
It is well-known that every $\Delta$-normed space $(E,\left\|\cdot\right\|)$ is metrizable  and conversely every  metrizable topological linear space can be equipped with a $\Delta$-norm  \cite[p.5]{KPR}.
In particular, when $C_E=1$, $E$ is called an \emph{$F$-normed} space \cite[p.3]{KPR}.
We note that every   $\Delta$-norm has an equivalent  $F$-norm \cite[Chapter 1.2]{KPR}.
We say that a  $\Delta$-norm $\left\|\cdot \right\|_\cA$ on a subspace $\cA$ of $S(0,1)$ is invariant with respect to translations, if for any translation $\alpha$ on $[0,1)$, we have $\alpha(\cA)\subseteq\cA$ (i.e., $\cA$ is translation-invariant) and  $\left\|\alpha(x)\right\|_\cA =\left\|x\right\|_\cA$.

We now come to the definition of the main object of this paper.
\begin{defn}
Let $\mathcal{E}$ be a linear subspace in $S(0,1)$
equipped with a $\Delta$-norm  $\left\|\cdot\right\|_{\mathcal{E}}$.
We say that
$\mathcal{E}$ is a \textit{symmetrically $\Delta$-normed  space}  if
for $x \in
\mathcal{E}$, $y\in S(0,1)$ and  $\mu(y)\leq \mu(x)$ imply that $y\in \mathcal{E}$ and
$\|y\|_\mathcal{E}\leq \|x\|_\mathcal{E}$ \cite{Astashkin,HS,HLS2017}.
Here, $\mu(f)$ stands for the decreasing rearrangement of $f\in S(0,1)$ \cite{LSZ,DP2}.
\end{defn}
Clearly,  symmetric $\Delta$-norms are invariant with respect to translation.
We note that convergence in the topology induced by any symmetric $\Delta$-norm  implies convergence in the measure topology on $S(0,1)$ \cite{Astashkin,HS,HLS2017}.
It is also known \cite{Astashkin,HS} that any symmetric $\Delta$-normed space contains all simple functions in $S(0,1)$.
In particular, if the symmetric $\Delta$-norm is order continuous (see \cite{HSZ18} for the definition), then all simple functions are dense in this symmetrically $\Delta$-normed  space \cite[Remark 2.9]{HSZ18}.

It is well-known \cite{HS,HM,Astashkin} that the $*$-algebra $S(0,1)$ can be equipped with a complete symmetric $\Delta$-norm.
Indeed,
by defining that
$$\left\|X\right\|_S = \inf_{t>0} [t+\mu(t;x)] ,~X\in S(0,1) ,$$
 we obtain a symmetric $F$-norm $\left\|\cdot\right\|_S$ on $S(0,1)$ \cite[Remark 3.4]{HS} (indeed, the constant for the quasi-triangle inequality is $1$).
Moreover, the topology induced by $\left\|\cdot\right\|_S$ is equivalent to the measure topology \cite[Proposition 4.1]{HS}.
Important examples of subalgebras of $S(0,1)$, such as
  $ L_\infty(0,1) $,    $ L_{\log}(0,1) $ \cite{DSZ2015} and the Arens algebra $ L^\omega (0,1) $ \cite{Arens},
can be equipped with complete symmetric  $F$-norms.

\subsection{Approximately differentiable functions}\label{AP}
Let us recall the concept of approximately differentiable functions.
Consider a Lebesgue measurable set $E\subset \mathbb{R},$
a measurable function $f: E \to \mathbb{R}$ and a point $t_0\in E,$  where $E$ has Lebesgue  density equal to $1.$ If the approximate limit
$$
f'_{ap}(t_0):=\textrm{ap}-\lim\limits_{t\to t_0}\frac{f(t)-f(t_0)}{t-t_0}
$$
exists and  it is finite, then it is called approximate derivative of the function $f$ at $t_0$ and the function is called approximately differentiable at $t_0$ (see \cite{Fed} for the details).
We note that all simple functions on $(0,1)$ are approximately differentiable.
However, it is clear that simple functions are not dense in $(AD(0,1),\rho)$.

\begin{rem}\label{1.1}	Let $AD(0,1)$ be the set of all classes $[f]\in S(0,1)$, for which $f$  have finite  approximate derivatives almost everywhere in $(0,1).$
 We note that the algebra $AD(0,1)$ is the $\rho$-completion   of the subalgebra in $S(0,1),$ generated by the algebra $C^{(1)}(0,1)$
of all continuously differentiable functions on $(0,1)$ and by the algebra  of all simple functions on $(0,1).$ Moreover, for any $x\in AD(0,1)$,  there exist a partition
of the unit
$\left\{\chi_{A_n}\right\}_{n\geq 1}$ and  a sequence $\{x_n\}_{n\geq 1}$ in $C^{(1)}(0,1)$  such that $\chi_{A_n}x=\chi_{A_n}x_n$ for all $n\geq1$ \cite[Proposition 4.7]{BKS1}.

Since the differential operator $\frac{d\cdot}{dt}$ commutes  with all dyadic-rational translations of the unit interval
$(0,1),$ it follows that the approximately differential operator   $\partial_{AD}:f \mapsto f'_{ap}$ also commutes with all $\alpha_n$, that is,
\begin{eqnarray*}
\partial_{AD} \circ \alpha_n =\alpha_n\circ\partial_{AD}.
\end{eqnarray*}
\end{rem}

In \cite{BKS1}, a  noncommutative analogue $AD(\cR)$
of the algebra of all almost everywhere approximately
differentiable functions on $(0,1)$ for the hyperfinite
type $II_1$ factor~$\cR$  was  introduced.
In particular, it is shown that
algebra $AD(\mathcal{R})$ is a proper subalgebra
of the Murray--von Neumann algebra $S(\mathcal{R})$ associated with $\cR$ and it is dense in the measure topology in  $S(\mathcal{R})$ but not in $\rho$-topology.
 Let $\mathcal{D}$ be the ``diagonal'' masa in~$\mathcal{R}.$
 We may identify $L_{\infty}(0,1)$ and $\mathcal{D}$
  via   a trace-preserving $*$-isomorphism $\pi$ of the algebras $L_{\infty}(0,1)$ and $\mathcal{D}$,
which uniquely   extends up to an one-to-one $*$-isomorphism between $ S(0,1)$ and $ S(\mathcal{D})$.
So, we may view $S(\mathcal{D})=S(0,1)$ as a $*$-subalgebra of~$S(\mathcal{R})$ \cite{BKS1}.
It is  established in \cite{BKS2} that the classical approximately differential operator on  $AD(0,1)$ admits an extension to a derivation $\delta$ from $AD(\mathcal{R})$ into $S(\mathcal{R})$ with $\delta|_{AD(0,1)} =\partial_{AD}$.

\section{Translation-invariance of derivations}
Let $\cA$ be a  $*$-algebra and
let
 $\delta$ be a derivation on   $\mathcal{A}.$
Set
$$
\delta_1(x)=\frac{\delta(x)+\delta(x^\ast)^\ast}{2},\,\, x\in  \mathcal{A}
$$
and
$$
\delta_2(x)=\frac{\delta(x)-\delta(x^\ast)^\ast}{2i},\,\, x\in  \mathcal{A}.
$$
Then, $\delta_1,$ $\delta_2$  are $\ast$-derivations (that is, $\delta_1(x^*)=\delta_1(x)^*$ and $\delta_2(x^*)=\delta_2(x)^*$) and $\delta=\delta_1+i \delta_2.$
Without loss of generality, from now on, we may assume that all derivations in this section  are $*$-derivations.
\subsection{The lack of extension of the  differential operator $\frac{d\cdot }{dt}$ up  to $S(\cR)$}

 Each element $a$ in an algebra $ \mathcal{A}$
implements a derivation $\textrm{ad}(a)$ on $\mathcal{A}$ defined as
$$\textrm{ad}(a)(x)=[a, x]=ax-xa,~x\in \mathcal{A}.$$
 Such derivations $\textrm{ad}(a)$
are called  \textit{inner derivations}.
 For  a detailed exposition of the theory of  derivations on operator algebras  we refer
to the monograph of Sakai~\cite{Sak2}.

It is known that every derivation on a von Neumann algebra $\cM$ is necessarily inner \cite{Kad66,Sak66}   (see \cite{Ber_S_1,Ber_S_2} for more general results for derivations with values into ideals of a von Neumann algebra).
However,
the properties of derivations of the
algebra $S(\mathcal{M})$ are far from being similar to those exhibited
by derivations on von Neumann algebras $\cM$.
In \cite{Ayupov}, Ayupov asked for  a full description of derivations on $S(\cM)$ (see also \cite{KL_survey,KL}, where Kadison and Liu asked a closely related question).
This long-standing open question has been treated in numerous papers (see e.g. \cite{AAK,BdPS,BCS14} and references therein).
The complete description for derivations on $S(\cM)$ for any von Neumann algebra  has been obtained  recently in  \cite{BKS2}.
In particular, for any type $II$ or $I_\infty$ von Neumann algebra $\cM$, derivations on $S(\cM)$ are automatically inner.
However,
for the  commutative von Neumann algebra $\mathcal{M}=L_\infty (0,1)$, the algebra $S(\mathcal{M})$  coincides with   $S(0,1)$, and the latter algebra admits non-trivial (and hence, non-inner) derivations \cite{BSCh04, BSCh06}.

In general, problems in the non-commutative setting are more complicated than their commutative counterparts.
However, due to the fact that there exist non-inner derivations on $S(0,1)$ \cite{BSCh06} and any derivations on $S(\cM)$ are inner when $\cM$ is a type $II$ (or $I_\infty$)  von Neumann algebra \cite{BKS2},
the proof yielding the lack of   extension of  differential operator up to $S(\cR)$ is much simpler than that for $S(0,1)$.

We now  present a proof of Proposition~\ref{1.2}.

\begin{proof}[Proof of Proposition~\ref{1.2}]
Let $\delta:S(\cM)\to S(\cM)$ be a derivation such that $\delta(\cP(\cA))\subset  S(\mathcal{A})$.
By \cite{BKS2},  $\delta$ is inner, in paricular, it is continuous in the measure topology. Since $\delta(\cP(\cA))\subset  S(\mathcal{A}),$
for any projection $e\in \cA$
we have
$\delta(e) = 0$.
Indeed, we have
	$$
	\delta(e)=\delta(e^2)=\delta(e)e+e\delta(e)=2e\delta(e).
	$$
	Multiplying the above equality  by $e$, we obtain $e\delta(e)=2e\delta(e).$ Hence, $e\delta(e)=0,$ and therefore $\delta(e)=0.$
 Thus, $\delta$ vanishes on the set of all linear combinations of mutually orthogonal  projections from $\cA.$
Since $\delta$ is continuous in measure  and  the set of all linear combinations of mutually orthogonal  projections  from $\cA$ is dense in measure in the real part $S(\cA)_h,$ it follows that $\delta$ also vanishes on $S(\cA)_h.$
By linearity of derivations, $\delta$   vanishes on $S(\cA)$.

Let $\cR$ be the hyperfinite type $II_1$ factor and $\cD$ be its diagonal masa.
Recall that  $\partial_{AD}$ maps $AD(\cD)(\cong AD(0,1))$ into $S(D)(\cong S(0,1))$ and  $P(\cD)\subset AD(\cD)$ (see Section \ref{AP}).
Let $\delta$ be a derivation on $S(\cR)$ as an extension of $\partial_{AD}$.
Setting $\cM =\cR$ and $\cA=\cD$, the first assertion of the proposition yields that $\partial_{AD}=\delta|_{AD(\cD)}=0$, which is a contradiction.
\end{proof}
 We note that for a derivation  vanishing  on the abelian subalgebra $\cA$, there exist non-trivial  extensions of this derivation on $S(\cM)$.
\begin{exam}Let $\cR$ be the hyperfinite $II_1$ factor and let $\cD$ be its diaganoal masa.
Recall  that every derivation $\delta:S(\cR)\to S(\cR)$ is implemented by an element in $S(\cR)$ \cite{BKS2}.
Let $a\in \cD\setminus \mathbb{C}{\bf 1}$.
We define a derivation $ad(a)(x)$, $x\in S(\cR)$.
In particular, $ad(a)=0$ on $\cD$ and therefore on $S(\cD)$.
However, it is non-trivial on $S(\cR)$.
Indeed, if $ad(a)$ is trivial on  $S(\cR)$, then $a$ is in $\cR'$. However, $\cR$ is a factor. Hence, $a\in \mathbb{C}{\bf 1}$, which is a contradiction with the assumption of $a$.
\end{exam}
 
\subsection{The proof of Theorem \ref{th11}}\label{3.2}

From now on, we concentrate on the   algebra  $S(0,1)$ of all classes of  Lebesgue measurable functions on $(0,1)$.

Recall that  $G $  is the group of all automorphisms of $S(0,1)$ generated by the dyadic translations of the unit interval $(0, 1),$ that is,
any element $\alpha$ of $G$ is defined as follows
\begin{align}\label{alphar}
\alpha(x)(t)=x\left(\left\{t-r\right\}\right),\,\,\, x\in S(0,1),\,\, t\in (0, 1),
\end{align}
where $r$ is the dyadic number from $[0, 1)$ and $\{a\}$ is the fractional part of the real number  $a$.
We say  that a derivation $\delta$ of $S(0,1)$ commutes  with $G$ or $G$-invariant, if
$$
\alpha \circ \delta =\delta \circ \alpha
$$
for all $\alpha\in G.$

The following theorem is the main result of the present section.

\begin{theorem}\label{thmmain}
Let $(\cA, \left\|\cdot\right\|_\cA)$ be a complete symmetric $\Delta$-normed   $*$-subalgebra of $S(0,1 )$
in which simple functions are dense in the $\left\|\cdot \right\|_\cA$-norm topology.
Let $\delta$ be a derivation from $\cA$ into $S(0,1 )$ commuting with $G.$  Then  $\delta$ is trivial.
\end{theorem}

From now on, we always assume that
 $\delta$ is a non-trivial  $*$-derivation from $\cA$ into  $S(0,1)$ commuting with $G.$
We construct below an element $h$ from $\cA $
such that  
 $\delta(h)\notin S(0,1)$. 
After that, we will present some properties of $h$,
which allow us to  show that
the $G$-invariance of $\delta$ fails at this element.

For any real-valued function $f\in S(0,1)$, we have $\delta(f)=\delta(f)^*$, that is, $\delta(f)$ is also a real function on $(0,1)$.
Since $\delta$ is a non-trivial $*$-derivation, there exists  an element $f\in \cA$ such that $\delta(f)\neq 0.$
If necessary, replacing $f$ with $-f$,  we can assume that the positive part of $\delta(f)$ is non zero.
Then we can find positive numbers $\lambda<\mu$ and  a measurable subset $A\subset \displaystyle (0,1) $ with a positive measure   such that
\begin{align*}
\lambda \chi_A\leq \chi_A \delta\left(f\right)\leq \mu  \chi_A.
\end{align*}
Note  that $\delta(e)=0$ for any projection $e\in \cA$ (see e.g. \cite[Prop. 2.3. (iii)]{BSCh06}).
Therefore,   $\delta\left(\chi_Af\right)=\chi_A \delta\left(f\right)$ and we obtain that
\begin{align*}
\lambda \chi_A\leq \delta\left(\chi_Af\right)\leq \mu  \chi_A.
\end{align*}
By replacing $f$ with $\displaystyle \frac{1}{\lambda}\chi_Af$,
we may  assume that $f$ is a function such that
\begin{align}\label{fff}
\chi_A\leq \delta\left(f\right)\leq \gamma
  \chi_A,
\end{align}
where $\gamma:=\frac{\mu}{\lambda }>1.$

Since there are countably many  dyadic numbers,
it follows that we can numerate all dyadic translations of $(0,1)$ as $\beta_n$, $n=0,1,\cdots$, i.e.,
 $G=\left\{\beta_n:n\geq 0\right\}$ ($\beta_0$ is an identical mapping on  $(0,1)$). Set
$$
B=\bigcup\limits_{n\geq 0}\beta_n(A).
$$
Observe that  for each $k\ge  0,$ we have  $\beta_k(B)=B $. 
Since the group  $G$ acts on  $(0,1)$ ergodically (see \cite[p. 927]{KRII}), it follows that
$m(B)=0$ or $m((0,1)\setminus B)=0.$ Since $m(B)>0,$ it follows that   \begin{align}\label{mb0}
m((0,1)\setminus B)=0.
\end{align}
Set
\begin{align*}
B_0 & =\beta_0(A)=A,\\
B_n &= \beta_n(A) \setminus\bigcup\limits_{k=0}^{n-1}\beta_k(A),~ n\geq 1.
\end{align*}
We note that all $B_n,\, n\ge 0,$ are pairwise  disjoint (note that $B_n$ may be the empty set for some $n$) and
$$
B=\bigcup_{n\geq 0} B_n.
$$
Taking into account into \eqref{mb0} and the last equality, we obtain that
$$
\sum_{n\geq 0} \chi_{B_n}=\chi_{[0,1)}.
$$
  Denote by $\widetilde{\beta}_k$ an automorphism of $S(0,1)$ generated by the translation $\beta_k,$ $k\ge 0$, that is, $\widetilde{\beta}_k(x)=x\circ \beta_k^{-1},\, x\in S(0,1).$  Set
\begin{align}\label{galphak}
g=\sum\limits_{k=0}^\infty \chi_{B_k}\widetilde{\beta}_k(f),
\end{align}
where the series is considered in $\rho$-topology (see Section \ref{MvN}).
Using~\eqref{fff} and the translation-invariance of $\delta$,  we have
  \begin{align}\label{BNBOUND}
  \chi_{B_n} \leq \chi_{B_n}\widetilde{\beta}_n\left(\delta(f)\right) =\chi_{B_n}  \delta(\widetilde{\beta}_n (f))  \leq \gamma\chi_{B_n }
   \end{align}
   for  all $n\geq 0.$
Taking into account that  $\delta$ is $\rho$-continuous \cite[Proposition 2.4]{BKS1}, we obtain  that
$$
\delta(g) \stackrel{\eqref{galphak}}{=} \delta\left(\sum\limits_{k=0}^\infty \chi_{B_k}\widetilde{\beta}_k(f)\right) =\sum _{k=0}^\infty \chi_{B_k} ( \delta(\widetilde{\beta}_k (f)))=\sum _{k=0}^\infty \chi_{B_k} \widetilde{\beta}_k( \delta(f)).
$$
Multiplying the above equality by $\chi_{B_n}$ and applying  \eqref{BNBOUND}, we have
$$\chi_{B_n}\le  \chi_{B_n} \delta(g )  =\chi_{B_n} \widetilde{\beta}_n(\delta(f)) \le \gamma \chi_{B_n} $$
Then, summing the above inequalities over all $n$, we obtain that
\begin{align}\label{1deltaggamma}
1 \leq \delta(g)\leq \gamma.
\end{align}

Due to  the assumption that simple functions are dense in $\cA$, for every $k\geq 1$,
there exists  a simple function $s_k\in S(0,1)$  such that
$\displaystyle    \left\| g-s_k \right \|_\cA\leq \frac{1}{(2\gamma)^{k^2}C_\cA^{2^{k+1}}}.$
Setting    $g_k:=g-s_k$,  we have that
\begin{align}\label{nn}
\left\|g_k \right\|_\cA  \leq \frac{1}{(2\gamma )^{k^2}C_\cA^{2^{k+1}}},
\end{align}
where $C_\cA$ is the constant for the quasi-triangle inequality for $\left\|\cdot \right\|_\cA$.
By \eqref{1deltaggamma} and  due to the fact that $\delta(s_k)=0$ (see e.g. \cite[Proposition 2.3,  (ii)]{BSCh06}), we obtain that
\begin{align}   \label{gk}
   1 \leq \delta(g_k)\leq \gamma, ~\forall k \ge 1.
\end{align}
Recall that for any $k\ge 1$, we have
defined automorphism $\alpha_k$ of $S(0,1)$ by
\begin{align*}
(\alpha_k(x))(t) =x\left(\left\{t - \frac{1}{2^k}\right\}\right).
\end{align*}
By the  definition of $\alpha_k$ (see \eqref{defalpha}), it is clear that
  $\alpha_n=\alpha_k^{2^{k-n}}$, $k>n$.
  Here,
we denote the decomposition of $\alpha_m$ with itself $i$-times by $\alpha_m^i $, that is, $\alpha_m^i=\alpha_m^{i-1}$ for $i\ge 2$ with $\alpha_m^1 =\alpha_m$.
Set
\begin{align}\label{hk}
h_k=\gamma^{k^2}\sum\limits_{i=1}^{2^k}(-1)^i\alpha_k^i\left(\chi_{\left[0,\frac{1}{2^k}\right)}g_k\right).
\end{align}
Appealing to the definition  of symmetric  $\Delta$-norms, we obtain that
$$
\left\|\alpha_k^i\left(\chi_{\left[0,\frac{1}{2^k}\right)}g_k\right)\right\|_\cA \le \left\|g_k \right\|_\cA  \stackrel{\eqref{nn}}{\leq} \frac{1}{(2\gamma)^{k^2}C_\cA^{2^{k+1}}},$$
for all $1\le i\le 2^k.$
Hence, by the quasi-triangle inequality, we obtain that
\begin{align*}
\left\| h_k \right \|_\cA  \leq C_\cA^{2^k}\sum _{i=1}^{2^k }\frac{ \gamma^{k^2} }{(2\gamma)^{k^2}C_\cA^{2^{k+1}}}=\frac{1}{2^{k^2-k } C_\cA^{2^k}}.
\end{align*}
Note that
$$\left\| \sum _{k=l}^n  h_k \right\|_\cA \le  \sum _{k=l}^n C_\cA ^ k\left\|  h_k \right\|_\cA \le       \sum _{k=l}^n C_\cA ^ k \frac{1}{2^{k^2-k } C_\cA^{2^k} } \le  \sum _{k=l}^n \frac{1}{2^{k^2-k  } } \to 0 \mbox{ as } l\to \infty. $$
Thus,   the   series
\begin{align}\label{hhh}
h:=\sum\limits_{k=1}^\infty h_k
\end{align}
 converges in $(\cA, \left\|\cdot \right\|_\cA)$.

Having constructed  the element $h\in S(0,1)$,
we shall now show  that $h$ is the required element,  at which   the $G$-invariance of $\delta$ fails.
Before proceeding to the proof of Theorem \ref{thmmain}, we collect some   relations between $h_k$ and $\alpha_n$, $k>n$.

From now on, the notations $h_k, k\in \mathbb{N}$, and $h$ always  stand for the functions defined in \eqref{hk} and \eqref{hhh}, respectively.

The following lemma shows that the elements $h_k$, $k\ge 1$, are well-behaved with respect to translations $\alpha_n$, $n\ge 1$.
\begin{lemma}
Let $n\ge 1$.
	Let $\alpha_n$ be  defined as  \eqref{defalpha} and $h_k$, $k\in \mathbb{N}$,  be as defined in \eqref{hk}. Then
	\begin{eqnarray}\label{aln}
	\alpha_n (h_n) =-h_n
	\end{eqnarray}
and
\begin{eqnarray}\label{alk}
	\alpha_n (h_k) =h_k
	\end{eqnarray}
	for all $k> n.$
\end{lemma}

\begin{proof}
When $k=n$, we have
\begin{align*}
\alpha_n(h_n) &\stackrel{\eqref{hk}}{=}  \gamma^{n^2}\sum\limits_{i=1}^{2^n}(-1)^i\alpha_n\left(\alpha_n^i\left(\chi_{\left[0,\frac{1}{2^n}\right)}g_n\right)\right)=
\gamma^{n^2}\sum\limits_{i=1}^{2^n}(-1)^i\alpha_n^{i+1}\left(\chi_{\left[0,\frac{1}{2^n}\right)}g_n\right)\\
 &~=
 -\gamma^{n^2}\sum\limits_{i=1}^{2^n}(-1)^{i+1}\alpha_n^{i+1}\left(\chi_{\left[0,\frac{1}{2^n}\right)}g_n\right)=-h_n.
\end{align*}
Recall that
  $\alpha_n=\alpha_k^{2^{k-n}}$, $k>n$.
  Hence, we obtain that
\begin{align*}
\alpha_n(h_k) &\stackrel{\eqref{hk}}{=}    \gamma^{k^2}\sum\limits_{i=1}^{2^k}(-1)^i\alpha_n\left(\alpha_k^i\left(\chi_{\left[0,\frac{1}{2^k}\right)}g_k\right)\right)=
 \gamma^{k^2}\sum\limits_{i=1}^{2^k}(-1)^i\alpha_k^{2^{k-n}+i}\left(\chi_{\left[0,\frac{1}{2^k}\right)}g_k\right)\\
 &~=
\gamma^{k^2}\sum\limits_{i=1}^{2^k}(-1)^{2^{k-n}+i}\alpha_k^{2^{k-n}+i}\left(\chi_{\left[0,\frac{1}{2^k}\right)}g_k\right)=h_k.
\end{align*}
\end{proof}

We provide below uniform  estimates  for the differences between $\alpha_n (\delta(h_k))$ and  $\delta(h_k) $, $n, k\ge 1$.
Recall that $\delta$ is a derivation on $S(0,1)$ commuting with $G$.
\begin{lemma}\label{forn}
	Let $n\geq 1$ be fixed. For every  $k\geq 1$,
 we have\footnote{We note that  when $k>n$, we have $\alpha_n ( \delta(h_k)) -\delta(h_k)=\delta(\alpha_n (h_k)) -\delta(h_k)\stackrel{\eqref{alk}}{=}0$.}
		\begin{eqnarray*}
\left|\alpha_n(\delta(h_k))-\delta(h_k)\right|\leq 2\gamma^{k^2+1},
	\end{eqnarray*}
and
		\begin{eqnarray}\label{x_n_q_n^1}
\left|\alpha_n(\delta(h_n))-\delta(h_n)\right|\geq 2\gamma^{n^2}.
	\end{eqnarray}
\end{lemma}

\begin{proof}
Recall that $\delta(e)=0$ for any projection $e\in \cA $ (see e.g. \cite[Prop. 2.3. (iii)]{BSCh06}, see also the proof of Proposition \ref{1.2}).
Using  \eqref{gk} we have
  \begin{align}\label{chi1}
  \chi_{\left[\frac{i}{2^k},\frac{i+1}{2^k}\right)} \leq \alpha_k^i\left(\delta(\chi_{\left[0,\frac{1}{2^k}\right)}g_k)\right)=
  \alpha_k^i\left(\chi_{\left[0,\frac{1}{2^k}\right)}\delta(g_k)\right)\leq \gamma\chi_{\left[\frac{i}{2^k},\frac{i+1}{2^k}\right)}
   \end{align}
   for all $i=1, \ldots, 2^k-1$ and
   \begin{align}\label{chi2}
  \chi_{\left[0,\frac{1}{2^k}\right)} \leq \alpha_k^{2^k}\left(\delta(\chi_{\left[0,\frac{1}{2^k}\right)}g_k)\right)\leq \gamma\chi_{\left[0,\frac{1}{2^k}\right)}.
   \end{align}
Since the derivation $\delta$ is  $G$-invariant, it follows that
 	\begin{align}\label{deltahk}
\delta(h_k)&\stackrel{\eqref{hk}}{=}\gamma^{k^2}\sum\limits_{i=1}^{2^k}(-1)^i\delta\left(\alpha_k^i\left(\chi_{\left[0,\frac{1}{2^k}\right)}g_k\right)\right)\nonumber\\
&=
\gamma^{k^2}\sum\limits_{i=1}^{2^k}(-1)^i\alpha_k^i\left(\delta\left(\chi_{\left[0,\frac{1}{2^k}\right)}g_k\right)\right)
	\end{align}
and, therefore, by \eqref{chi1} and \eqref{chi2}, we have
\begin{eqnarray*}
		\left|\delta(h_k)\right|\leq \gamma^{k^2+1}.
	\end{eqnarray*}
Hence, $\left|\alpha_n(\delta(h_k))\right|\leq \gamma^{k^2+1}.$
By the triangle inequality, we obtain that
\begin{eqnarray*}
		\left|\alpha_n(\delta(h_k))-\delta(h_k)\right|\leq 2\gamma^{k^2+1}.
	\end{eqnarray*}
On the other hand, by \eqref{deltahk}, summing \eqref{chi1}  over $i=1,2,\cdots,2^{n-1}$
 and \eqref{chi2}, we have
\begin{eqnarray*}
		\left|\delta(h_n)\right|\ge  \gamma^{n^2}.
	\end{eqnarray*}
Taking \eqref{aln} into account, i.e., $\alpha_n(h_n)=-h_n $,
we obtain that
\begin{eqnarray*}
		\left|\alpha_n(\delta(h_n))-\delta(h_n)\right|=\left|\delta(\alpha_n(h_n))-\delta(h_n)\right|= \left|-2 \delta(h_n)\right|\geq 2\gamma^{n^2},
	\end{eqnarray*}
which completes the proof.
  \end{proof}

In the next lemma,  we provide  a uniform estimate  for the sum of the difference between $\delta(h_k)$ and $\alpha_n (\delta(h_k))$ over all $k=1,\dots, n$.
\begin{lemma}\label{enenen}
There exists a number $N$ such that for any $n\ge N$, we have
	\begin{eqnarray}\label{2n}
		\left|\sum\limits_{k=1}^n \alpha_n(\delta(h_k))-\delta(h_k)\right| \geq \gamma^{n^2}.
	\end{eqnarray}
\end{lemma}

\begin{proof} By Lemma \ref{forn}, we have
	\begin{equation*}
	\left|\alpha_n(\delta(h_k))-\delta(h_k)\right| \leq 2\gamma^{k^2+1}.
	\end{equation*}
Recall that $\gamma >1$ and 	note that for sufficiently large $n$, we have
	\begin{eqnarray*}
		\sum\limits_{k=1}^{n-1}\gamma^{k^2+1} & \leq (n-1) \gamma^{(n-1)^2+1}\leq \gamma^n\cdot \gamma^{(n-1)^2+1}=\gamma^{n^2-n+2}.
	\end{eqnarray*}
 We infer that
	\begin{equation}\label{en_k<n-1}
\left| 	\sum\limits_{k=1}^{n-1} \alpha_n(\delta(h_k))-\delta(h_k)\right| \leq\sum\limits_{k=1}^{n-1} \left|\alpha_n(\delta(h_k))-\delta(h_k)\right| \leq 2\sum_{k=1}^{n-1} \gamma^{k^2+1} \le 2\gamma^{n^2-n+2}.
	\end{equation}
When  $k=n$, by  \eqref{x_n_q_n^1},  we have that
	\begin{eqnarray}\label{enen}
	\left|\alpha_n(\delta(h_n))-\delta(h_n)\right| \geq  2\gamma^{n^2}.
	\end{eqnarray}
Combining inequalities \eqref{en_k<n-1} and \eqref{enen}, we conclude that
	\begin{align*}
\left|\sum\limits_{k=1}^n \alpha_n(\delta(h_k))-\delta(h_k)\right|
&	\geq 	
\left|\alpha_n(\delta(h_n))-\delta(h_n)\right|	-	\left| \sum\limits_{k=1}^{n-1} \alpha_n(\delta(h_k))-\delta(h_k)\right|\\
& \geq 2\gamma^{n^2}-2\gamma^{n^2-n+2}=\gamma^{n^2} + (\gamma^{n^2} - 2\gamma^{n^2-n+2}) \\
&=\gamma^{n^2} + \gamma^{n^2-n+2}(\gamma^{n-2 } - 2)\ge \gamma^{n^2}
	\end{align*}
for all sufficiently large $n$.
\end{proof}

The following lemma is the key ingredient  in our proof. It shows that an estimate similar to that of Lemma \ref{enenen} holds for the infinite sum $h=\sum_{k=1}^\infty h_k$.

\begin{lemma}\label{ffff}
There exists a number $N$ such that for any $n\ge N $, we have
	\begin{eqnarray}\label{nondiffer}
	\left|\alpha_n(\delta(h))-\delta(h)\right| & \geq  &  \gamma^{n^2}.
	\end{eqnarray}
\end{lemma}

\begin{proof}
Let $N$ be large enough such that \eqref{2n} holds and let $n\ge N$.
Recall that $\left\|\cdot \right\|_\cA$-convergence implies measure-convergence (see Section \ref{2.2}).
Recall  the symmetric $\Delta$-norm  $\left\|\cdot\right\|_S$  on $S(0,1)$ induced by the measure topology (see Section \ref{2.2}).
Hence, $\left \|  \sum _{k=l}^\infty h_k \right\|_\cS  \stackrel{\eqref{hhh}}{\rightarrow} 0 $  as $l\to \infty$.
For any $l>n,$ we have
\begin{align*}
&\qquad \left\|\alpha_n\left(\sum _{k=n+1}^\infty h_k\right) -  \sum _{k=n+1}^\infty h_k \right\|_S\\
&~\le~ \left \| \alpha_n\left(\sum _{k=n+1}^l h_k\right) -  \sum _{k=n+1}^l h_k \right\|_S +  \left \| \alpha_n \left(\sum_{k=l}^\infty h_k\right)\right\|_S + \left \|  \sum _{k=l}^\infty h_k \right\|_S  \\
&\stackrel{\eqref{alk}}{=} 2 \left \|  \sum _{k=l}^\infty h_k \right\|_\cS   \rightarrow 0 ~ \mbox{ as $l\to \infty$.} \end{align*}
Hence, $ \alpha_n \left(\sum\limits_{k=n+1}^\infty h_k\right) = \sum\limits_{k=n+1}^\infty h_k$.
By the assumption, $\delta$ is $G$-invariant and hence it commutes with $\alpha_n$, $n\ge 1$. This implies  that for any $n \ge N$, we have
	\begin{align*}
	\left|\alpha_n(\delta(h))-\delta(h)\right| &~ =  \left|\delta(\alpha_n(h))-\delta(h)\right|=
\left|\delta(\alpha_n(h)-h)\right|\\
&\stackrel{\eqref{hhh}}{=} \left|\delta\left(\alpha_n\left(\sum\limits_{k=n+1}^\infty h_k\right)-\sum\limits_{k=n+1}^\infty h_k\right)+\delta\left(\sum\limits_{k=1}^{n} (\alpha_n(h_k)-h_k)\right)\right|\\
& ~= ~ \left|\delta\left(\sum\limits_{k=1}^{n} (\alpha_n(h_k)-h_k)\right)\right|=
\left|\sum\limits_{k=1}^{n} \alpha_n(\delta(h_k))-\delta(h_k)\right|\stackrel{\eqref{2n}}{\geq} \gamma^{n^2}.
	\end{align*}
Here, the series $\sum\limits_{k=n+1}^\infty \alpha_n(h_k)$ and $\sum\limits_{k=n+1}^\infty h_k$ are considered in the topology with respect to $\left\|\cdot \right\|_\cA$ and therefore, in the measure topology.
	\end{proof}

The following lemma is a simple  observation. For the sake of completeness, we incorporate  a detailed  proof for it.
\begin{lemma}\label{c2}Let  $F:\mathbb{N}\to \mathbb{R}$ be an  increasing function with $\lim_{n\to \infty} F(n)=\infty $.
For any $y\in S(0,1)$,  the inequality
$$
|\alpha_n(y) - y|  \geq  F(n)
$$
fails for all  sufficiently large number $n.$
\end{lemma}

\begin{proof}
Take a closed subset $A$ in $(0,1)$ with the Lebesque measure
$\displaystyle m(A)>\frac{3}{4}$ such that
$$
|y|\chi_{A}\leq c
$$
for some $c>0.$
Further, for each $n\geq 1$, the  closed subset $A_n: =\left\{\left\{t+\frac{1}{2^n}\right\}: t\in A\right\}$ satisfies
$\displaystyle m(A_n)=m(A)>\frac{3}{4}$ and
$$
|\alpha_n(y)|\chi_{A_n}=|\alpha_n(y)|\alpha_n\left(\chi_{A}\right)=\alpha_n\left(|y|\chi_{A}\right)\leq c.
$$
For each $n\ge 1$,  we have that $\displaystyle m(A\cap A_n)>\frac{1}{2}$ and
$$
  |\alpha_n(y)-y|\chi_{A\cap A_n}\leq 2c\chi_{A\cap A_n}.
$$
Assume by contradiction that  $F(n)\le |\alpha_n(y) - y|   $ for all sufficiently large   $n$.
 Then, $$F(n)  \chi_{A\cap A_n}\leq |\alpha_n(y)-y|\chi_{A\cap A_n}\le 2c\chi_{A\cap A_n},$$
which implies that   $F(n) \leq 2c$ for all sufficiently large $n$.
 This is  a contradiction with the assumption on the function $F$.
\end{proof}

Now we are in a position to present the proof of our main result, Theorem \ref{thmmain}.

\begin{proof}[Proof of Theorem \ref{thmmain}]
Assume by contradiction that there exists a non-trivial $\delta:\cA \to S(0,1)$ commuting with $G$.
	Let $h$ be defined as in  \eqref{hhh}.
By Lemma \ref{ffff},  the function  $y:=\delta(h)\in \cA $
satisfies inequality  \eqref{nondiffer}.
 Setting $F(n)=\gamma^{n^2}$, $n\in \mathbb{N}$, we obtain  a contradiction with  Lemma~\ref{c2}.
Hence, there exists no non-trivial  derivation $\delta:\cA \to S(0,1)$ commuting with $G$.
\end{proof}

The following result is an immediate consequence of Theorem \ref{thmmain}, which shows that the translation invariance property of $\frac{d\cdot}{dt}$ is not shared by its extension.

\begin{corollary}\label{cor:ado}
The  differential operator $\frac{d\cdot }{dt}$ has no translation-invariant extension to the  algebra $S(0,1)$.
\end{corollary}




 Theorem \ref{thmmain} holds for   a subalgebra $(\cA,\left\|\cdot \right\|_\cA)$ of $S(0,1)$ in which simple functions are dense.
However, simple functions are not dense in $(AD(0,1),\rho)$.
By Theorem~\ref{thmmain}, for  any complete symmetric $\Delta$-normed subalgebra $(\cA,\left\|\cdot \right\|_\cA)$ of $S(0,1)$ in which  simple functions are dense, there exists no derivation from $\cA$ into $S(0,1)$ commuting with all dyadic-rational translations on $(0,1)$.
Also, recall that it is shown in \cite{BKS1} that the algebra $AD(0,1)$ is the maximal subalgebra of $S(0,1)$ admitting unique extension of $\frac{d\cdot}{dt}:D(0,1)\to S(0,1)$.
It is interesting to drop  the ``symmetric  $\Delta$-normed'' assumption and  consider the following problem.
\begin{prob}
	Is the algebra $AD(0,1)$ a maximal subalgebra in $S(0,1)$ admitting a translation-invariant derivation as an  extension of   $\frac{d\cdot}{dt}:D(0,1)\to S(0,1)$?
\end{prob}

{\bf Acknowledgements}
The   third   author was supported by the Australian Research Council  (FL170100052).

 \end{document}